\pdfoutput=1
\documentclass[oneside,leqno,11pt]{amsart}
\usepackage{amsthm}
\usepackage{amsmath, amssymb, amsfonts}
\usepackage{graphicx}
\usepackage[textwidth=16cm,textheight=20cm]{geometry}
\usepackage{hyperref}
\usepackage{bbm}

\newtheorem{thm}[equation]{Theorem}

\newtheorem{lem}[equation]{Lemma}

\newtheorem{prop}[equation]{Proposition}
\theoremstyle{definition}

\numberwithin{equation}{section}

\def\P{{\mathbb{P}}}
\def\R{{\mathbb{R}}}
\def\E{{\mathbb{E}}}
\def\8{\infty}
\renewcommand{\a}{\alpha}

\renewcommand{\le}{\leqslant}\renewcommand{\leq}{\leqslant}
\renewcommand{\geq}{\geqslant}

\DeclareMathOperator*{\argmin}{arg\,min}

\begin{document}

\title{Precise large deviations  of the first passage time}
\author[D. Buraczewski, M. Ma\'slanka]
{Dariusz Buraczewski, Mariusz Ma\'slanka}
\address{D. Buraczewski, M. Ma\'slanka\\ Instytut Matematyczny\\ Uniwersytet Wroclawski\\ 50-384 Wroclaw\\
pl. Grunwaldzki 2/4\\ Poland}
\email{dbura@math.uni.wroc.pl\\ maslanka@math.uni.wroc.pl}

\subjclass[2010]{Primary 60G50, 60F10}

\keywords{first passage time, ruin problem, large deviations, random walk.}

\thanks{The research was partially supported by the National Science Centre, Poland (Sonata Bis, grant number  UMO-2014/14/E/ST1/00588)}

\begin{abstract}
Let $S_n$ be partial sums of an i.i.d.~sequence $\{X_i\}$. We assume that $\E X_1 <0$ and $\P[X_1>0]>0$. In this paper we study  the first passage time
$$ \tau_u = \inf\{n:\; S_n > u\}.
$$
The classical Cram\'er's estimate of the ruin probability says that
$$
\P[\tau_u<\infty] \sim C e^{-\a_0 u}\qquad \mbox{as } u\to \8,
$$  for some parameter $\a_0$.
 The aim  of the paper is to describe precise large deviations of the first crossing by $S_n$ a linear boundary, more precisely for a fixed parameter $\rho$ we study asymptotic behavior of  $\P\big[\tau_u = \lfloor u/\rho\rfloor \big]$ as $u$ tends to infinity.
\end{abstract}

\maketitle

\section{Introduction}
Let $\{X_i\}$ be a sequence of independent and identically distributed (i.i.d.) real valued random variables. We denote by $S_n$ the partial sums of $X_i$, i.e. $S_0=0$, $S_n = X_1+\cdots+ X_n$. In this paper we are interested in the situation when $X_1$ has negative drift, but simultaneously $\P[X_1 >0]>0$.  Our primary objective is to describe the precise large deviations of the linearly normalized first passage time
$$\tau_u = \inf \{n:S_n > u \},$$
as $u$ tends to infinity.

The stopping time $\tau_u$ arises in various contexts in probability, e.g. in risk theory, sequential statistical analysis, queueing theory.   
We refer to Siegmund \cite{S} and Lalley \cite{Lalley} for a comprehensive bibliography. A celebrated result concerning $\tau_u$, playing a major role in the ruin theory, is due to Cram\'er, who revealed estimate of the ruin probability
\begin{equation}
\label{eq: cramer}
\P[\tau_u<\infty] \sim C e^{-\alpha_0 u}, \qquad \mbox{as } u\to\infty,
\end{equation} for some parameter $\alpha_0$ that  will be described below (see Cram\'er \cite{C} and Feller \cite{F}).

Our aim is to describe the probability that at a given time partial sums $S_n$ first cross a linear boundary $\rho n$. This problem was studied e.g. by Siegmund \cite{S} and  continued by Lalley \cite{Lalley}.
Up to our best knowledge all the known results concern probabilities of the form $\P[\tau_u< u/\rho ]$ or $\P[u/\rho < \tau_u <\infty]$, see Lalley \cite{Lalley} (see also Arfwedson \cite{Arf} and Asmussen \cite{Asmussen} for similar results related to compound Poisson risk model). In this paper we describe pointwise behavior of $\tau_u$, i.e. the asymptotic behavior of $\P\big[\tau_u = \lfloor u/\rho  \rfloor \big]$ as $u$ tends to infinity.

\section{Statement of the results}

Our main result will be expressed in terms of  the moment and cumulant generating functions of $X_1$, i.e.
\[
\lambda(s) = \mathbb{E}[e^{sX_1}] \quad \text{and} \quad  \Lambda(s) = \log \lambda(s),
\]
respectively. We assume that $\lambda(s)$ exists for $s$ in the interval $D = [0, s_0)$ for some $s_0 > 0$. It is well known that both $\lambda$ and $\Lambda$ are smooth and convex on $D$.
Throughout the paper we assume that there are $\alpha \in D$ and $\xi > 0$ such that
\begin{equation}\label{eq:1.1}
\rho= \Lambda'(\alpha) > 0
\end{equation}
and
\begin{equation*}
 \lambda(\alpha + \xi) < \infty.
\end{equation*}
Observe that \eqref{eq:1.1} implies that $\mathbb{P}\left[X_1 > 0 \right] > 0$.

Recall the convex conjugate (or the Fenchel-Legendre transform) of $\Lambda$ defined by
$$
\Lambda^*(x) = \sup_{s\in \R}\{sx - \Lambda(s)\}, \quad x\in\R.
$$ This rate function appears in studying large deviations problems for random walks. Its various properties can be found in Dembo, Zeitouni \cite{DZ}. Given $\a < s_0$ and $\rho$ as in \eqref{eq:1.1} we consider
$$
\overline \a = \frac 1{\rho}\; \Lambda^*(\rho).
$$ An easy calculation shows
\[
\overline{\alpha} = \alpha - \frac{\Lambda(\alpha)}{\Lambda'(\alpha)}.
\]
The parameter $\overline \a$ arises in the classical large deviations theory for random walks. The Petrov's theorem and the Bahadur-Rao theorem say that
\begin{equation}\label{eq: petrov}
\P[S_n > n\rho] \sim C \; \frac{e^{-\overline \a n \rho}}{\sqrt n} \qquad \mbox{as } n \to \infty,
\end{equation}
 (see Petrov \cite{Petrov} and Dembo, Zeitouni \cite{DZ}). As we will see below $\overline \a$ will play also the crucial role in our result.
 This parameter has a geometric interpretation: the tangent line to $\Lambda$ at point $\alpha$ intersects the
$x$-axis at $\overline{\alpha}$. See the Figure \ref{fig} below.

\begin{figure}[!h]
\centering
\caption{$\Lambda(s) = \log\mathbb{E}e^{sX_1}$}\label{fig}
\includegraphics{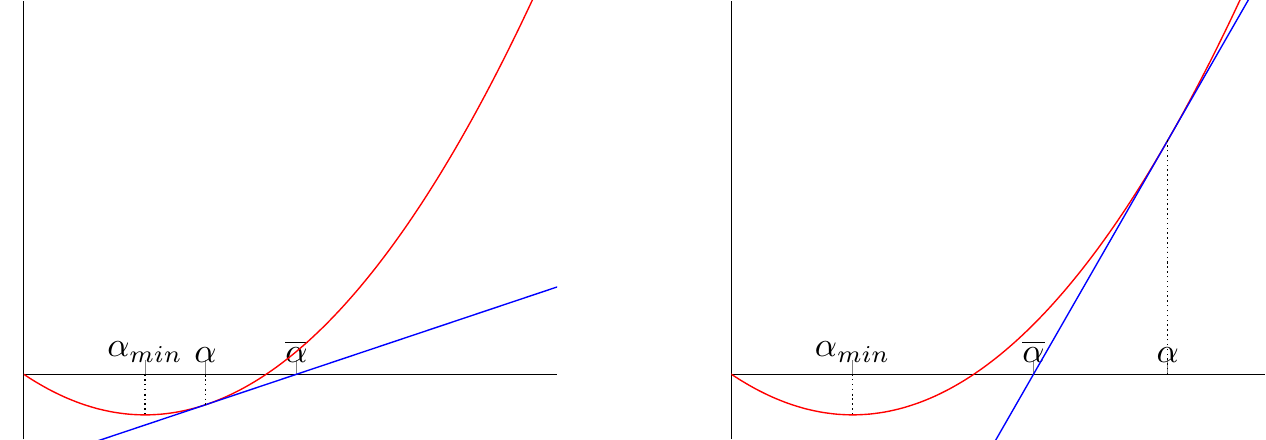}
\end{figure}

We also introduce parameters  $k_u$ and $\alpha_{min}$ defined by
\[
 \alpha_{min} = \argmin \Lambda(s)  \quad \text{and } \quad  k_u = \frac{u}{\rho} .
\]

\medskip

Now we are ready to state our main result.

\begin{thm}\label{th1} Assume that $\{X_i\}$ is an i.i.d. sequence such that the law of $X_1$ is nonlattice, $\E X_1<0$ and $\rho = \Lambda'(\a)>0$ for some $\a<s_0$. Then
\begin{equation*}
\begin{split}
\mathbb{P}\left[\tau_u =\left \lfloor{k_u }\right \rfloor \right] =  C(\alpha) \lambda(\alpha)^{-\Theta (u)} \; \frac{e^{-u\overline{\alpha}}}{\sqrt{u}} \; (1+o(1)) \quad \text{as} \quad  u \to \infty
\end{split}
\end{equation*}
for some constant $C(\a)>0$ and  $\Theta (u) = k_u - \left \lfloor{k_u}\right \rfloor$.
\end{thm}
Notice that the above formula gives the largest asymptotics when $\a=\a_0$ for $\a_0$ such that $\Lambda(\a_0)=0$. Then $\overline \a_0 = \a_0$. For all the other parameters $\a$ we have $\overline \a > \overline \a_0$.
 The parameter $\a_0$  arises in the Cram\'er's  formula \eqref{eq: cramer}.

Similar results were obtained by Lalley, who proved that  for $\alpha$ such that $\Lambda(\alpha) > 0$ we have
\begin{equation*}
\mathbb{P} \left[\tau_u \leq k_u \right] = {C_1(\alpha) \lambda(\alpha)^{-\Theta (u)}}\; \frac{e^{-u \overline{\alpha}}}{\sqrt{u}} \; (1+o(1)) \quad \text{as} \quad  u \to \infty
\end{equation*}
and for $\alpha$ such that $\Lambda(\alpha) < 0$
\begin{equation*}
\mathbb{P} \left[\tau_u > k_u \right] = {C_2(\alpha) \lambda(\alpha)^{1-\Theta (u)}} \; \frac{ e^{-u \overline{\alpha}}}{\sqrt{u}}\; (1+o(1)) \quad \text{as} \quad  u \to \infty,
\end{equation*}
for some known, depending only on $\alpha$ constants $C_1(\alpha)$, $C_2(\alpha)$ (see Lalley \cite{Lalley}, Theorem 5).

Notice that the function $\Theta(u)$ appears in all the formulas above only from purely technical reason. It reflects the fact that $\tau_u$ attains only integer values, whereas $k_u$ is continuous. Thus the function $\Theta$ is needed only to adjust both expressions for noninteger values of $k_u$. Below we will omit this point and without any saying we assume that $k_u$ is an integer.

\section{Auxillary results.}

The proof of Theorem \ref{th1} bases on the Petrov's theorem and the Bahadur-Rao theorem describing precise large deviations for random walks \eqref{eq: petrov}.
We apply here techniques, which were recently used by Buraczewski et al. \cite{BCDZ, BDZ} to study the problem of the first passage time in a more general context of perpetuities. They obtained similar results as described above, but in our context the proof is essentially simpler and final results are stronger.
%

Here we need a reinforced version of \eqref{eq: petrov}, which is both uniform and allows to slightly  perturb the parameters. As a direct consequence of Petrov's theorem \cite{Petrov} the following results was proved in \cite{BCDZ}:
\begin{lem}\label{lem: petrov}
Assume that the law of $X_1$ is nonlattice and that $\rho$ satisfies $\mathbb{E} X_1 < \rho< A_0$.
Choose $\alpha$ such that $\Lambda'(\alpha) = \rho $. If $\{\delta_n\}$, $\{j_n\}$ are two sequences  satisfying
\begin{equation}\label{eq10}
\max \{\sqrt{n}  \left| \delta_n  \right|, j_n/\sqrt{n} \} \leq \overline{\delta}_n \to 0,
\end{equation}
then
\begin{equation*}
\mathbb{P}\left[S_{n-j_n} > n\left(\rho + \delta_n\right) \right] = C(\alpha)\frac{e^{-\overline{\alpha} n\rho}}{\sqrt{n}}  e^{ - \alpha n \delta_n} \lambda(\alpha)^{-j_n }(1+o(1)) \quad \quad \text{as}\ n \to \infty,
\end{equation*}
uniformly with respect to $\rho$ in the range
\begin{equation*}
\mathbb{E}X + \epsilon \leq \rho \leq A_0 - \epsilon,
\end{equation*}
and for all $\delta_n$, $j_n$ as in \eqref{eq10}.

\end{lem}

Let us define $M_n = \max_{1 \leq k \leq n}S_k$ and $S_{i}^n = S_n - S_{n-i} =  X_{n - i +1}+...+X_{n}$ for $0 \leq i \leq n$. The following Lemma will play a crucial role in the proof.

\begin{lem}\label{l1}
Let $L$ and $M$ be two integers such that  $L \geq 1$ and  $-1 \leq M \leq L$. For any $ \gamma \geq 0$, $\alpha_{min} < \beta < \alpha$ and sufficiently large $u$, the following holds
\begin{equation*}
\begin{split}
 \mathbb{P}\left[M_{k_u - L} > u, S_{k_u-M} > u - \gamma \right] & \leq C(\alpha, \beta) e^{\gamma \beta} \lambda(\alpha)^{-L} \lambda(\beta)^{L - M}   \frac{e^{-u\overline{\alpha}}}{\sqrt{ u}},
\end{split}
\end{equation*} where $C(\a,\beta)$ is some constant depending on $\a$ and $\beta$.
\end{lem}

\begin{proof}
We have
\begin{equation*}
\begin{split}
 \mathbb{P}\left[M_{k_u - L} > u, S_{k_u-M} > u - \gamma \right] & \leq \sum_{i=0}^{k_u -1 -L} \mathbb{P}\left[ S_{k_u-M} > u - \gamma , S_{k_u - i - L} > u\right].
\end{split}
\end{equation*}
Denote $\delta = \frac{\lambda(\beta)}{\lambda(\alpha)} < 1$. To estimate the above series, we divide the set of indices into two sets.

\noindent
{\sc Case 1.} First we consider $i$ satisfying
 $i > K \log k_u $ for some constant $K$ such that $\delta^{K \log k_u} < 1/u$.
Notice that for any $u$ we have
\begin{equation*}
e^{-u \overline{\alpha}} = e^{-u\alpha} \lambda({\alpha})^{k_u}.
\end{equation*}
 Then, for any such $i$ we write
\begin{equation*}
\begin{split}
\mathbb{P}\!\left[ S_{k_u-M} \!>\! u \!- \!\gamma , S_{k_u - i - L} \!>\! u\right] & \leq
 \sum_{m=0}^{\infty} \mathbb{P}\!\left[S_{k_u-M} \!>\! u \!-\! \gamma ,u\!+\!m \!<\! S_{k_u - i - L} \!\leq\! u \!+\! m\!+\!1\right]\\
 & = \sum_{m=0}^{\infty} \mathbb{P}\!\left[S_{k_u-i-L} \!+\!  S_{L+i-M}^{k_u-M}  \!>\! u \!-\! \gamma, u \!+\! m \!<\! S_{k_u - i - L} \!\leq\! u \!+\! m\!+\!1\right] \\
 & \leq \sum_{m=0}^{\infty} \mathbb{P}\!\left[ S_{L+i-M}^{k_u-M} \!>\! -\gamma\!-\!(m\!+\!1)\right] \mathbb{P}\!\left[S_{k_u - i - L} \!>\! u\!+\!m\right] \\
  & \leq \sum_{m=0}^{\infty} e^{\beta \gamma} e^{\beta(m+1)} \lambda(\beta)^{L+i-M} e^{-u\alpha} e^{-\alpha m} \lambda(\alpha)^{k_u - i - L}\\
  & \leq C(\alpha, \beta) e^{\beta \gamma} \delta^{i} e^{-u\overline{\alpha}} \lambda(\alpha)^{-L} \lambda(\beta)^{L-M},
\end{split}
\end{equation*}
where in the third line we used Markov's inequality with functions $e^{\beta x}$ and  $e^{\alpha x}$. Summing over $i$ we obtain

\begin{equation*}
\begin{split}
 \sum_{K \log k_u < i \leq k_u -1 + L} \!\!\!\mathbb{P}\left[ S_{k_u-M} > u-\gamma , S_{k_u - i - L} > u\right] & \leq C(\alpha, \beta) \sum_{ i > K \log k_u}\!\!  e^{\beta \gamma} \delta^{i} e^{-u\overline{\alpha}} \lambda(\alpha)^{-L} \lambda(\beta)^{L-M} \\
 & \leq  C(\alpha, \beta) e^{\beta \gamma} \delta^{K \log k_u} e^{-u\overline{\alpha}}  \lambda(\alpha)^{-L} \lambda(\beta)^{L-M} \\
 & \leq  C(\alpha, \beta) e^{\beta \gamma} \frac{e^{-u\overline{\alpha}}}{ u}  \lambda(\alpha)^{-L} \lambda(\beta)^{L-M}.
\end{split}
\end{equation*}

\noindent
{\sc Case 2.} Now consider $ i \leq K \log k_u$. Let $N$ be a constant such that $-\alpha N +1 < 0$, for $\Lambda(\alpha) \geq 0$ and $-\alpha N +1 - \Lambda(\alpha)K < 0$ for $\Lambda(\alpha) < 0$. We have
\begin{equation*}
\begin{split}
\mathbb{P}\left[ S_{k_u-M} > u-\gamma , S_{k_u - i - L} > u\right] & \leq \mathbb{P}\left[ S_{k_u - i -L} \geq u + N \log  k_u \right] \\ &\quad + \mathbb{P}\left[S_{k_u-M} > u-\gamma, u< S_{k_u - i -L} < u +N \log  k_u\right]\\
& = P_1 + P_2
\end{split}
\end{equation*}
The first term $P_1$ we estimate using Markov's inequality with function $e^{\alpha x}$ and we obtain
\begin{equation*}
\begin{split}
P_1 & \leq e^{-u\alpha} k_u^{-\alpha N} \lambda(\alpha)^{k_u-i-L}  =  e^{-u\overline{\alpha}} k_u^{-\alpha N} \lambda(\alpha)^{-i} \lambda(\alpha)^{-L}  \leq C(\alpha) e^{-u\overline{\alpha}} \frac{1}{u} k_u^{-\alpha N + 1} e^{-i \Lambda(\alpha)}  \lambda(\alpha)^{-L} \\
& \leq C(\alpha) e^{-u\overline{\alpha}} \frac{1}{ u} \lambda(\alpha)^{-L}.
\end{split}
\end{equation*} To estimate $P_2$ we apply Lemma \ref{lem: petrov} and again Markov's inequality with function $e^{\beta x}$.
\begin{equation*}
\begin{split}
P_2 & =\mathbb{P}\left[S_{k_u-i-L} + S_{L+i-M}^{k_u-M}  > u-\gamma, u< S_{k_u - i -L} < u + N \log k_u \right] \\
& \leq \sum_{m=0}^{\left \lceil{N \log k_u - 1}\right \rceil } \mathbb{P}\left[S_{k_u-i-L} + S_{L+i-M}^{k_u-M} > u-\gamma, u + m < S_{k_u - i -L} < u +m+1\right]\\
& \leq \sum_{m=0}^{\left \lceil{N \log k_u - 1}\right \rceil } \mathbb{P}\left[S_{k_u-i-L} > u+m \right] \mathbb{P}\left[S_{L+i-M}^{k_u-M} > -\gamma -(m+1) \right]\\
& \leq \sum_{m=0}^{\left \lceil{N \log k_u - 1}\right \rceil } C(\alpha) \frac{e^{-u\overline{\alpha}}}{\sqrt{k_u}} \lambda(\alpha)^{-i-L} e^{-\alpha m} e^{\beta \gamma} e^{\beta (m+1)} \lambda(\beta)^{i+L-M}\\
& \leq \sum_{m=0}^{\left \lceil{N \log k_u - 1}\right \rceil } C(\alpha, \beta) \frac{e^{-u\overline{\alpha}}}{\sqrt{k_u}} e^{(\beta-\alpha) m} \delta^i e^{\beta \gamma}  \lambda(\alpha)^{-L} \lambda(\beta)^{L-M}  \\
& \leq  C(\alpha, \beta) \frac{e^{-u\overline{\alpha}}}{\sqrt{u}}  \delta^i e^{\beta \gamma}  \lambda(\alpha)^{-L} \lambda(\beta)^{L-M}.
\end{split}
\end{equation*}
 Now we sum over $i$
 
\begin{equation*}
\begin{split}
 \sum_{ i \leq K \log k_u}  \mathbb{P}[ S_{k_u-M} > u - \gamma &, S_{k_u - i - L} > u] \\
 &\leq \sum_{ i \leq K \log k_u} \!\!\! \left( P_1 + P_2 \right) \\
 & \leq  \sum_{ i \leq K \log k_u} \!\!\! \left( C(\alpha)  \frac{e^{-u\overline{\alpha}}}{u} \lambda(\alpha)^{-L} + C(\alpha, \beta) \frac{e^{-u\overline{\alpha}}}{\sqrt{u}}  \delta^i e^{\beta \gamma}  \lambda(\alpha)^{-L} \lambda(\beta)^{L-M} \right) \\
 & \leq C(\alpha) e^{-u\overline{\alpha}} \frac{ \log k_u}{u} \lambda(\alpha)^{-L} + C(\alpha, \beta) \frac{e^{-u\overline{\alpha}}}{\sqrt{ u}} e^{\beta \gamma} \lambda(\alpha)^{-L} \lambda(\beta)^{L-M} .
\end{split}
\end{equation*}
Combining both cases we end up with
\begin{equation*}
\begin{split}
 \mathbb{P}\left[M_{k_u - L} > u, S_{k_u-M} >  u - \gamma \right]  & \leq C(\alpha, \beta) e^{\beta \gamma} \frac{e^{-u\overline{\alpha}}}{u}  \lambda(\alpha)^{-L} \lambda(\beta)^{L-M}   +  C(\alpha) e^{-u\overline{\alpha}} \frac{ \log k_u}{u} \lambda(\alpha)^{-L}\\
 &\quad + C(\alpha, \beta) \frac{e^{-u\overline{\alpha}}}{\sqrt{u}}   e^{\beta \gamma} \lambda(\alpha)^{-L} \lambda(\beta)^{L-M} \\
 & \leq C(\alpha, \beta) \frac{e^{-u\overline{\alpha}}}{\sqrt{u}}   e^{\beta \gamma} \lambda(\alpha)^{-L} \lambda(\beta)^{L-M}.
\end{split}
\end{equation*}
\end{proof}

\section{Lower and upper estimates}

The goal of this section is to prove the following
\begin{prop}\label{prop1}
There is a constant $C > 0$ such that for large $u$
\begin{equation}\label{pr1eq1}
\begin{split}
\frac {1} C \;  \frac{e^{-u\overline{\alpha}} }{\sqrt{u}} \le
\mathbb{P}\left[\tau_u =  k_u + 1 \right]
\le C 
 \;  \frac{e^{-u\overline{\alpha}} }{\sqrt{u}}.
\end{split}
\end{equation}
\end{prop}
\begin{proof}
First, observe that the upper estimate  is an immediate consequence of Petrov's theorem (Lemma \ref{lem: petrov}) used with $\gamma_n = 0$. Indeed, we have
\begin{equation*}
\mathbb{P}\left[\tau_u = k_u + 1\right] = \mathbb{P}\left[M_{k_u} \leq u, S_{k_u+1} > u\right] \leq \mathbb{P}\left[ S_{k_u+1} > u\right]  \leq  C(\alpha) \frac{e^{-u\overline{\alpha}}}{\sqrt{u}}.
\end{equation*}
For the lower estimate we write for any positive $\gamma$ and any positive integer $L$
\begin{equation*}
\begin{split}
\mathbb{P}\left[\tau_u = k_u + 1\right]  = \mathbb{P}\left[M_{k_u} \leq u, S_{k_u+1} > u\right]  \geq \mathbb{P}\left[M_{k_u} \leq u, S_{k_u+1} > u, S_{L+1}^{k_u+1} > \gamma \right] .
\end{split}
\end{equation*}
For any $0 < r < \gamma$ one has
\begin{equation*}
\begin{split}
  \mathbb{P}\left[M_{k_u} \leq u, S_{k_u+1} > u, S_{L+1}^{k_u+1} > \gamma \right]  \geq \mathbb{P}\left[M_{k_u} \leq u, u - \gamma < S_{k_u-L} < r+u-\gamma, S_{L+1}^{k_u+1} > \gamma \right].
\end{split}
\end{equation*}
Let \!$M_i^n \!\!=\!\! \max(0,S_{1}^{n-i+1},S_{2}^{n-i+2},S_{3}^{n-i+3},...,S_{i-1}^{n-1},S_{i}^{n})$.\! Note that $M_{k_u} \!\!\!\!=\!\max(\!M_{k_u - L},S_{k_u - L}\!+\! M_L^{k_u})$. Hence we have
\begin{equation*}
\begin{split}
\mathbb{P}[M_{k_u} \leq u, u - \gamma & < S_{k_u-L} < r+u-\gamma, S_{L+1}^{k_u +1} > \gamma] \\
&  = \mathbb{P}\left[M_{k_u-L} \leq u,S_{k_u - L} + M_{L}^{k_u} \leq u, u - \gamma < S_{k_u-L} < r+u-\gamma, S_{L+1}^{k_u +1} > \gamma \right]\\
&  \geq \mathbb{P}\left[M_{k_u-L} \leq u,M_{L}^{k_u} \leq -r+\gamma, u -\gamma < S_{k_u-L} < r+u-\gamma, S_{L+1}^{k_u +1} > \gamma \right].
\end{split}
\end{equation*}
Finally, we combine above, use independence of $(M_{L}^{k_u}, S_{L+1}^{k_u +1})$ and $(M_{k_u-L}, S_{k_u-L})$ and the identity ${\mathbb{P}[A \cap B] = \mathbb{P}[A] -\mathbb{P}[A \cap B^c]}$ to obtain
\begin{equation}\label{p3}
\begin{split}
  \mathbb{P}\!\left[\tau_u \!=\! k_u \!+\! 1\right] & \geq \mathbb{P}\!\left[M_{k_u-L} \leq u,M_{L}^{k_u} \leq -r\!+\!\gamma, u \!-\!\gamma < S_{k_u-L} < r\!+\!u\!-\!\gamma, S_{L+1}^{k_u +1} > \gamma \right]\\
  & = \mathbb{P}\!\left[M_{k_u-L} \leq u, u\!-\!\gamma < S_{k_u-L} < r\!+\!u\!-\!\gamma \right] \mathbb{P}\!\left[ M_{L}^{k_u} \leq -r\!+\!\gamma, S_{L+1}^{k_u +1} > \gamma \right] \\
  & = \mathbb{P}\!\left[M_{L}^{k_u} \leq -r\!+\!\gamma ,  S_{L+1}^{k_u +1} > \gamma \right]\\
  & \quad \times \left( \mathbb{P}\!\left[u\!-\!\gamma < S_{k_u-L} < r\!+\!u\!-\!\gamma\right] - \mathbb{P}\!\left[M_{k_u-L} > u, u\!-\!\gamma < S_{k_u-L} < r\!+\!u\!-\!\gamma\right] \right).
\end{split}
\end{equation}

Lemma \ref{lem: petrov} gives an asymptotics
\begin{equation}\label{p1}
\begin{split}
\mathbb{P}\left[u-\gamma < S_{k_u-L} < r+u-\gamma\right] & \sim  C(\alpha, r) e^{\alpha \gamma} \lambda(\alpha)^{-L}  \frac{e^{-u\overline{\alpha}}}{\sqrt{u}} \quad \quad \text{as } u \to \infty.
\end{split}
\end{equation}

Using Lemma \ref{l1} with $M = L$ we obtain
\begin{equation}\label{p2}
\begin{split}
 \mathbb{P}\left[M_{k_u - L} > u, u-\gamma < S_{k_u-L} < r+u-\gamma \right] & \leq \mathbb{P}\left[M_{k_u - L} > u, S_{k_u-L} > u-\gamma \right] \\
 & \leq C(\alpha, \beta) \frac{e^{-u\overline{\alpha}}}{\sqrt{u}}   e^{\beta \gamma} \lambda(\alpha)^{-L},
\end{split}
\end{equation}
where $\beta < \alpha$. From \eqref{p3}, \eqref{p1} and \eqref{p2} we have
\begin{equation*}
\begin{split}
\mathbb{P}\!\left[\tau_u \!=\! k_u \!+\! 1\right] & \geq \mathbb{P}\!\left[ S_{L+1}^{k_u +1} \!>\! \gamma, M_{L}^{k_u} \leq -r \!+\! \gamma\right]\\
  & \quad \times \left( \mathbb{P}\!\left[u\!-\!\gamma < S_{k_u-L} < r\!+\!u\!-\!\gamma \right] - \mathbb{P}\!\left[M_{k_u-L} > u, u\!-\!\gamma < S_{k_u-L} < r\!+\!u\!-\!\gamma\right] \right) \\
&\geq \mathbb{P}\!\left[ S_{L+1}^{k_u +1} \!>\! \gamma, M_{L}^{k_u} \leq -r\!+\!\gamma \right]\! \left(\! C(\alpha, r)  \lambda(\alpha)^{-L}  \frac{e^{-u\overline{\alpha}}}{\sqrt{u}} e^{\alpha \gamma} \!-  C(\alpha, \beta) \frac{e^{-u\overline{\alpha}}}{\sqrt{u}}   e^{\beta \gamma} \lambda(\alpha)^{-L} \!\right)\\
& = \mathbb{P}\left[ S_{L+1}^{k_u +1} \!>\! \gamma, M_{L}^{k_u} \leq -r+\gamma\right]\lambda(\alpha)^{-L} \frac{e^{-u\overline{\alpha}}}{\sqrt{u}} \left( C(\alpha, r) e^{\alpha \gamma} -  C(\alpha, \beta)   e^{\beta \gamma}  \right).
\end{split}
\end{equation*}

Notice that $\left(M_{i}^{n}, S_{i+1}^{n+1} \right)  \,{\buildrel d \over =}\, \left(M_i, S_{i+1} \right)$. To make constants in the last term strictly positive firstly pick $r > 0$ such that $\mathbb{P}\left[X_1 > 2r\right] > 0$. Next, take $\gamma > 0$ big enough to ensure that $C(\alpha, r)     e^{\alpha \gamma} -  C(\alpha, \beta)   e^{\beta \gamma} > 0$ and $\gamma - 2r > 0$.
Now we choose large $L$ to have $\mathbb{P}\left[L X_1 >-2r +\gamma\right] > 0$. Since $\gamma$ is continuous parameter, if necessary, we can increase it to get ${\mathbb{P}\left[-2r +\gamma < L X_1 < - r +\gamma \right] > 0}$. For such constants we have
\begin{equation*}
\begin{split}
0 & < \mathbb{P}\left[X_{L+1} > 2r \right] \prod_{i=1}^{L} \mathbb{P}\left[-2r +\gamma < L X_i < - r +\gamma \right]\\
& \leq \mathbb{P}\left[X_{L+1} > 2r, S_L \geq S_{L-1} \geq ... \geq S_1,-2r + \gamma < S_L < -r+\gamma \right]\\
& \leq \mathbb{P}\left[X_{L+1} > 2r, S_L = M_L,-2r + \gamma < S_L < -r+\gamma \right]\\
& \leq \mathbb{P}\left[M_L < -r+\gamma, S_{L+1} > \gamma \right],
\end{split}
\end{equation*}
and \eqref{pr1eq1} follows.
\end{proof}

\section{Asymptotics}
\begin{proof}[Proof of Theorem \ref{th1}]
We will show that the limit
\begin{equation}\label{eq5}
\lim_{u \to \infty} e^{u\overline{\alpha}}\sqrt{u}\,  \mathbb{P}\left[\tau_{u} = k_u +1  \right]
\end{equation}
exists, which combined with Proposition \ref{prop1} gives us Theorem \ref{th1}. \newline
Fix an arbitrary $L$. Since $M_{k_u} = \max(M_{k_u - L},S_{k_u - L}+ M_{L}^{k_u})$ we have
\begin{equation}\label{eq6}
\begin{split}
 \mathbb{P}\left[ S_{k_u -L } + M_{L}^{k_u} \leq u, S_{k_u +1}  >u \right]  & =  \mathbb{P}\left[ S_{k_u -L } +M_{L}^{k_u} \leq u, S_{k_u +1}  >u, M_{k_u - L} > u \right]\\& \quad + \mathbb{P}\left[ S_{k_u -L }+ M_{L}^{k_u} \leq u, S_{k_u +1}  >u, M_{k_u - L} \leq u \right] \\
 & = \mathbb{P}\left[ S_{k_u -L } +M_{L}^{k_u} \leq u, S_{k_u +1}  >u, M_{k_u - L} > u \right]  \\& \quad + \mathbb{P}\left[M_{k_u} \leq u, S_{k_u +1}  >u \right] \\
  & = \mathbb{P}\left[ S_{k_u -L }+ M_{L}^{k_u} \leq u, S_{k_u +1}  >u, M_{k_u - L} > u \right]\\& \quad + \mathbb{P}\left[\tau_{u} = k_u +1  \right].
\end{split}
\end{equation}
From Lemma \ref{l1} with $M = -1$ and $\gamma = 0$ we obtain
\begin{equation*}
\mathbb{P}\left[M_{k_u - L} > u, S_{k_u+1} > u \right]  \leq C(\alpha, \beta) \lambda(\alpha)^{-L} \lambda(\beta)^{L+1}  \frac{e^{-u\overline{\alpha}}}{\sqrt{u}} = C(\alpha, \beta) \delta^{L}  \frac{e^{-u\overline{\alpha}}}{\sqrt{u}},
\end{equation*}
where $\delta = \frac{\lambda(\beta)}{\lambda(\alpha)} < 1$ provided $\beta < \alpha$. Thus to get \eqref{eq5} it is sufficient to show that for some large fixed $L$
\begin{equation*}
\lim_{u \to \infty} e^{u\overline{\alpha}}\sqrt{u}\, \mathbb{P}\left[ S_{k_u -L }+ M_{L}^{k_u} \leq u, S_{k_u +1}  >u \right]
\end{equation*}
exists. Indeed, multiply both sides of \eqref{eq6} by $e^{u\overline{\alpha}}\sqrt{u}$, let first $u \to \infty$ and then $L \to \infty$. We write
\begin{equation*}
\begin{split}
 \mathbb{P}\left[ S_{k_u -L }+ M_{L}^{k_u} \leq u, S_{k_u +1}  >u \right] = \quad&
 \mathbb{P}\left[u -u^{\frac{1}{4}} < S_{k_u - L} < u, S_{k_u -L } +M_{L}^{k_u} \leq u, S_{k_u +1} >u \right] \\
 &  + \mathbb{P}\left[u -u^{\frac{1}{4}} \geq S_{k_u - L}, S_{k_u -L }+ M_{L}^{k_u} \leq u, S_{k_u +1} >u \right].
\end{split}
\end{equation*}

To estimate the second summand fix $\beta > \a$ and   observe that by Markov's inequality with functions $e^{\alpha x}$ and $e^{\beta x}$ we have
\begin{equation*}
\begin{split}
 \mathbb{P}[S_{k_u - L} \leq u -u^{\frac{1}{4}}&, S_{k_u +1} >u ] \\
 &\leq \sum_{m \geq 0} \mathbb{P}\left[u -u^{\frac{1}{4}} -(m+1) < S_{k_u - L} \leq u -u^{\frac{1}{4}} -m, S_{k_u - L}+S_{L+1}^{k_u +1} >u \right]\\
 & \leq \sum_{m \geq 0} \mathbb{P}\left[S_{k_u - L}  > u -u^{\frac{1}{4}} -(m+1)\right] \mathbb{P}\left[S_{L+1} > u^{\frac{1}{4}}+ m \right] \\
 & \leq \sum_{m \geq 0}  \lambda(\alpha)^{k_u - L} e^{-u\alpha}  e^{\alpha u^{\frac{1}{4}}} e^{\alpha(m+1)} \lambda(\beta)^{L+1} e^{-\beta u^{\frac{1}{4}}} e^{-\beta m } \\
 & =  \lambda(\alpha)^{k_u - L} e^{-u\alpha}  e^{(\alpha - \beta)u^{\frac{1}{4}}} \lambda(\beta)^{L+1} \sum_{m \geq 0} e^{\alpha(m+1)} e^{-\beta m } = o\left( \frac{e^{-u\overline{\alpha}}}{\sqrt{u}} \right).
\end{split}
\end{equation*}
The same argument proves
$$
\P\big[ S_{k_u-L} > u-u^{\frac 14}, S^{k_u+1}_{L+1} > u^{\frac 14}\big] = o\left( \frac{e^{-u\overline{\alpha}}}{\sqrt{u}} \right).
$$
Now we see that
\begin{equation*}
\begin{split}
 \mathbb{P}[ S_{k_u -L }+& M_{L}^{k_u} \leq u, S_{k_u +1}  >u ] \\ 
 & = \mathbb{P}\left[u -u^{\frac{1}{4}} < S_{k_u - L} < u, S_{k_u -L } +M_{L}^{k_u} \leq u, S_{k_u +1} >u \right] + o\left( \frac{e^{-u\overline{\alpha}}}{\sqrt{u}} \right)\\
& =  \mathbb{P}\Big[u -u^{\frac{1}{4}} < S_{k_u - L} < u , S_{k_u -L } +M_{L}^{k_u} \leq u, S_{k_u +1} >u, S_{L+1}^{k_u+1} < u^{\frac 14} \Big] + o\left( \frac{e^{-u\overline{\alpha}}}{\sqrt{u}} \right)
\end{split}
\end{equation*}
and hence we reduced our problem to finding
\begin{equation*}
\lim_{u \to \infty} e^{u\overline{\alpha}}\sqrt{u}\, \mathbb{P}\left[u -u^{\frac{1}{4}} < S_{k_u - L} < u, S_{k_u -L }+ M_{L}^{k_u} \leq u, S_{k_u +1} >u, S_{L+1}^{k_u+1} < u^{\frac 14} \right].
\end{equation*}
For this purpose we write
\begin{equation}\label{eq8}
\begin{split}
 &\mathbb{P}\left[u  -u^{\frac{1}{4}} < S_{k_u - L} < u, S_{k_u -L }+ M_{L}^{k_u}\leq u, S_{k_u - L}+ S_{L+1}^{k_u +1} >u, S_{L+1}^{k_u+1} < u^{\frac 14} \right] \\
 &=  \int_{0\le y\le x < u^{\frac 14}}  \mathbb{P}\left[u - x < S_{k_u - L} < u - y \right]  \mathbb{P} \left[M_{L}^{k_u} \in dy, S_{L+1}^{k_u +1} \in dx  \right].
\end{split}
\end{equation}
Now we apply Lemma \ref{lem: petrov}   with $n=k_u$, $j_n = L$, $\overline{\delta}_n = C n^{-\frac{1}{4}}$ and  $\delta_n = -\frac{y}{n}$. We have
$$    \mathbb{P}\left[S_{k_u - L} \geq u - y  \right]  = C(\alpha) \frac{e^{-u\overline{\alpha}}}{\sqrt{u}} e^{y \alpha} e^{-L \Lambda(\alpha)} (1+o(1)),$$
provided $\max \Big\{\frac{\sqrt{u}}{u}  y,\, L/\sqrt{u} \Big\} \leq C u^{-\frac{1}{4}}$. But since $y < u^{\frac{1}{4}}$ all the assumptions of the Lemma are satisfied.
Analogously
$$    \mathbb{P}\left[S_{k_u - L} \geq u - x  \right]  = C(\alpha) \frac{e^{-u\overline{\alpha}}}{\sqrt{u}} e^{x \alpha} e^{-L \Lambda(\alpha)} (1+o(1)).$$
Back to \eqref{eq8} we end up with
\begin{equation*}
\begin{split}
 &\mathbb{P}\left[u  -u^{\frac{1}{4}} < S_{k_u - L} < u, S_{k_u -L }+ M_{L}^{k_u}\leq u, S_{k_u - L}+ S_{L+1}^{k_u +1} >u \right] \\
 &= C(\alpha) \frac{e^{-u\overline{\alpha}}}{\sqrt{u}} e^{-L \Lambda(\alpha)} \mathbb{E}\left[\left( e^{\alpha S_{L+1}} - e^{\alpha M_{L}} \right)_{+} \right]   (1+o(1)) \quad \text{as } u \to \infty.
\end{split}
\end{equation*}
Note that by the moment assumptions the expectation above  is finite, hence we conclude \eqref{eq5}.
\end{proof}

\end{document}